\documentclass{amsart}

\newtheorem{theorem}{Theorem}[section]

\theoremstyle{definition}

\newtheorem{proposition}[theorem]{Proposition}
\newtheorem{corollary}[theorem]{Corollary}

\theoremstyle{remark}

\numberwithin{equation}{section}

\begin{document}

\title[  Inequalities for eigenvalues of fourth order elliptic operators ]{ Inequalities for eigenvalues of fourth order elliptic operators in divergence form on Riemannian manifolds }

\author{Shahroud Azami}
\address{Department of Mathematics, Faculty of Sciences, Imam Khomeini International University, Qazvin, Iran. }

\email{azami@sci.ikiu.ac.ir}






\begin{abstract}
In this paper, we study eigenvalue of  linear fourth order elliptic operators in divergence form  with  Dirichlet boundary condition on  a bounded domain in a compact Riemannian manifolds with boundary (possibly empty) and find  a general inequality for them. As an application, by using this inequality, we study eigenvalues of this operator on compact domains  of complete submanifolds in  a Euclidean space.
\end{abstract}

\maketitle
\section{Introduction}
In this paper, let $(M, \langle,\rangle)$  be an $n $-dimensional complete Riemannian manifold and let $\Omega \subset M$ be  a bounded connected domain with smooth boundary  $\partial \Omega$ in $M$. Denote by $\Delta$ the Beltrami-Laplace operator on $M$. The study of the  spectrum of  geometric operator is an important  topic and many works  have been done in this area.  The clamped plate problem or the Dirichlet biharmonic operator  for a connected bounded domain  $\Omega \subset \mathbb{R}^{n}$ is given by
\begin {equation}\label{1}
\begin{cases}
\Delta^{2} u=\lambda u,& \text{in}\,\,\, \Omega\\
u=\frac{\partial u}{\partial \nu}=0&\text{on}\,\,\,\partial \Omega,
\end{cases}
\end{equation}
where $\nu$ is the outward unit normal vector field of  $\partial \Omega$. Suppose that $\{\lambda_{i}\}_{i=1}^{\infty}$ is eigenvalues of the problem (\ref{1}). Payn  et'al proved in paper \cite{P} that
\begin{equation*}
\lambda_{k+1}-\lambda_{k}\leq \frac{8(n+2)}{n^{2}}\frac{1}{k}\sum_{i=1}^{k}\lambda_{i},\,\,\,\,\,\,\,\,k=1,2,....
\end{equation*}
In 1984, Hile and Yeh \cite{HY} generalized and showed
\begin{equation*}
\sum_{i=1}^{k}\frac{\lambda_{i}^{\frac{1}{2}}}{\lambda_{k+1}-\lambda_{i}}\geq \frac{n^{2}k^{\frac{3}{2}}}{8(n+2)}\left(\sum_{i=1}^{k}\lambda_{i}\right)^{-\frac{1}{2}},\,\,\,\,\,\,\,\,k=1,2,....
\end{equation*}
In 1990, Hook \cite{H}obtained the following inequality
\begin{equation*}
 \frac{n^{2}k^{2}}{8(n+2)}\leq\left(\sum_{i=1}^{k}\lambda_{i}^{\frac{1}{2}}\right)
\left(\sum_{i=1}^{k}\frac{\lambda_{i}^{\frac{1}{2}}}{\lambda_{k+1}-\lambda_{i}}\right),\,\,\,\,\,\,\,\,k=1,2,....
\end{equation*}
In 2006, Cheng and Yang \cite{C} obtained  the  inequality
\begin{equation*}
\lambda_{k+1}-\frac{1}{k}\sum_{i=1}^{k}\lambda_{i}\leq
\left(\frac{8(n+2)}{n^{2}}\right)^{\frac{1}{2}}\frac{1}{k}
\sum_{i=1}^{k}\left(\lambda_{i}(\lambda_{k+1}-\lambda_{i})\right)^{\frac{1}{2}}
,\,\,\,\,\,\,\,\,k=1,2,....
\end{equation*}
In 2007, Wang and Xia \cite{W}, proved  universal bounds for eigenvalues of  the biharmonic operator on Riemannian manifolds, for instance they showed that, when $\Omega$ is a compact domain in $\mathbb{R}^{n}$, then
\begin{equation*}
\lambda_{k+1}\leq \frac{1}{k}\sum_{i=1}^{k}\lambda_{i}+\left\{
\frac{64}{n^{2}k^{2}}(\sum_{i=1}^{k}\lambda_{i}^{\frac{1}{2}})(\sum_{i=1}^{k}\lambda_{i}^{\frac{3}{2}})
-\frac{1}{k}\sum_{i=1}^{k}(\lambda_{i}-\frac{1}{k}\sum_{i=1}^{k}\lambda_{i})^{2}
\right\}^{\frac{1}{2}},\,\,k=1,2,....
\end{equation*}
The aim of the present work is to study the  eigenvalues of linear fourth  order elliptic operator  in divergence form on Riemannian manifolds. In special case, this operator is the biharmonic operator. We prove some general inequalities for them. By using  these  inequalities, we obtain, when $\Omega$ is a compact  domains of complete submanifolds in a  Euclidean space.\\

Let $T$  be symmetric positive definite $(1,1)$-tensor on $M$ and  $\Omega \subset M$ be a compact  domain with  smooth boundary $\partial \Omega$ in $M$. We will studying  the eigenvalue problem
\begin {equation}\label{2}
\begin{cases}
\mathcal{L}^{2} u=\lambda u,& \text{in}\,\,\, \Omega\\
u=\frac{\partial u}{\partial \nu}=0&\text{on}\,\,\,\partial \Omega,
\end{cases}
\end{equation}
where $\mathcal{L} u=div(T(\nabla u))$ and $\nabla$ is the gradient operator of $M$. If $T$ be the identity tensor the $\mathcal{L}=\Delta$.
The main results of this paper are  as follow\\

\begin{theorem}\label{t1} Let $\Omega$ be a domain in an $n$-dimensional complete Riemannian manifold  $(M, \langle,\rangle)$ isometrically  immersed in $\mathbb{R}^{m}$, $\lambda_{i}$ be  the $i$th eigenvalue of  (\ref{2}) and $u_{i}$  be the corresponding orthonormal real-valued eigenfunction, that is
\begin {equation}\label{3}
\begin{cases}
\mathcal{L}^{2} u_{i}=\lambda_{i} u_{i},& \text{in}\,\,\, \Omega\\
u_{i}=\frac{\partial u_{i}}{\partial \nu}=0,&\text{on}\,\,\,\partial \Omega\\
\int_{\Omega}u_{i}u_{j}dm=\delta_{ij}&\forall i,j =1,2,....
\end{cases}
\end{equation}
Then for any positive constant $\delta$ and any positive integer $k$, we have
\begin{equation}\label{4}
\sum_{i=1}^{k}(\lambda_{k+1}-\lambda_{i})^{2}\int_{\Omega}u_{i}^{2}tr(T)dm\leq \delta \sum_{i=1}^{k}(\lambda_{k+1}-\lambda_{i})^{2}A_{i}+\frac{1}{\delta}
\sum_{i=1}^{k}(\lambda_{k+1}-\lambda_{i})B_{i}
\end{equation}where
\begin{eqnarray*}
A_{i}&=&2\int_{\Omega}u_{i}\mathcal{L}u_{i}\langle tr(\alpha\circ T)+tr(\nabla T), I\rangle dm+2\int_{\Omega}u_{i}\langle T(\nabla \mathcal{L}u_{i}), I\rangle dm\\&&+\int_{\Omega}u_{i}^{2}(||tr(\alpha\circ T)||^{2}+|tr(\nabla T)|^{2})dm+4\int_{\Omega}u_{i}\langle T(\nabla u_{i}), tr(\nabla T)\rangle dm\\&&+4\int_{\Omega}|T(\nabla u_{i})|^{2}dm+2\int_{\Omega}\mathcal{L}u_{i}\langle T(\nabla u_{i}),I\rangle dm,
\end{eqnarray*}
and
\begin{equation*}
B_{i}=\int_{\Omega}\left\{ |T(\nabla u_{i})|^{2}+u_{i}\langle T(\nabla u_{i}), tr(\nabla T)\rangle+\frac{u_{i}^{2}}{4}(||tr(\alpha\circ T)||^{2}+|tr(\nabla T)|^{2})\right\}dm,
\end{equation*}
where $I(x)=(x_{1},...,x_{m})$ for any $x=(x_{1},...,x_{m})\in\mathbb{R}^{n}$, $||f||^{2}=\int_{\Omega}f^{2}dm$,  $dm$ is the volume form on $\Omega$, $\alpha$ is the fundamental form of $M$ and $\alpha \circ T=\alpha(T(.),.)$.
\end{theorem}

\begin{theorem}\label{t2} Let $\Omega$ be a domain in an $n$-dimensional complete Riemannian manifold  $(M, \langle,\rangle)$ isometrically  immersed in $\mathbb{R}^{m}$, $\lambda_{i}$ be  the $i$th eigenvalue of  (\ref{2}) and $u_{i}$  be the corresponding orthonormal real-valued eigenfunction.
Then for any positive constant $\delta$ and any positive integer $k$, we have
\begin{equation}\label{4}
\sum_{i=1}^{k}(\lambda_{k+1}-\lambda_{i})^{2}\int_{\Omega}u_{i}^{2}tr(T)dm\leq \delta \sum_{i=1}^{k}(\lambda_{k+1}-\lambda_{i})^{2}C_{i}+\frac{1}{\delta}
\sum_{i=1}^{k}(\lambda_{k+1}-\lambda_{i})D_{i}
\end{equation}where
\begin{eqnarray*}
C_{i}&=&2(\sqrt{m-n} S_{0}T_{*}+T_{0})I_{0}\lambda_{i}^{\frac{1}{2}}+I_{0}||T(\nabla\mathcal{L}u_{i})||_{L^{2}(\Omega)}+ (m-n)S_{0}^{2}T_{*}^{2}+T_{0}^{2}\\&&+4 T_{0}||T(\nabla u_{i})||_{L^{2}(\Omega)}+4 ||T(\nabla u_{i})||_{L^{2}(\Omega)}+2\lambda_{i}
 I_{0}||T(\nabla u_{i})||_{L^{2}(\Omega)}
\end{eqnarray*}
and
\begin{equation*}
D_{i}=||T(\nabla u_{i})||_{L^{2}(\Omega)}+T_{0}||T(\nabla u_{i})||_{L^{2}(\Omega)}+\frac{1}{4}((m-n)S_{0}^{2}T_{*}^{2}+T_{0}^{2})
\end{equation*}
where $S_{0}=\max\{\sup_{\bar{\Omega}}|S_{e_{k}}|: \,\,k=n+1,...,m\}$, $S_{e_{k}}$ is the Weingarten operator of the immersion with respect to $e_{k}$, $T_{*}=\sup_{\bar{\Omega}} |T|$, $T_{0}=\sup_{\bar{\Omega}}|tr(\nabla T)|$ and $I_{0}=\sup_{\bar{\Omega}}|I|$.
\end{theorem}

\begin{theorem}\label{t3}
Let $\Omega$ be a domain in an $n$-dimensional complete Riemannian manifold  $(M, \langle,\rangle)$ isometrically  immersed in $\mathbb{R}^{m}$ with mean curvature $H$ and  $\lambda_{i}$ be  the $i$th eigenvalue of  biharmonic operator, that is
\begin {equation*}
\begin{cases}
\Delta^{2} u_{i}=\lambda_{i} u_{i},& \text{in}\,\,\, \Omega\\
u_{i}=\frac{\partial u_{i}}{\partial \nu}=0,&\text{on}\,\,\,\partial \Omega\\
\int_{\Omega}u_{i}u_{j}dm=\delta_{ij}&\forall i,j =1,2,....
\end{cases}
\end{equation*}
Then for any positive constant $\delta$ and any positive integer $k$, we have
\begin{eqnarray}\nonumber
n\sum_{i=1}^{k}(\lambda_{k+1}-\lambda_{i})^{2}&\leq& \delta \sum_{i=1}^{k}(\lambda_{k+1}-\lambda_{i})^{2}(2nH_{0}I_{0}\lambda_{i}^{\frac{1}{2}}+n^{2}H_{0}^{2}+4\lambda_{i}^{\frac{1}{2}})\\\label{d1}&&
+\frac{1}{\delta}
\sum_{i=1}^{k}(\lambda_{k+1}-\lambda_{i})(\lambda_{i}^{\frac{1}{2}}+\frac{1}{4}n^{2}H_{0}^{2})
\end{eqnarray}
where $H_{0}=\sup_{\bar{\Omega}}|H|$ and $I_{0}=\sup_{\bar{\Omega}}|I|$.
\end{theorem}

\begin{corollary}\label{t4}
Let $\Omega$ be a domain in an $n$-dimensional complete minimal Riemannian submanifold  $(M, \langle,\rangle)$ in $\mathbb{R}^{m}$ and  $\lambda_{i}$ be  the $i$th eigenvalue of  biharmonic operator.
Then for any  positive integer $k$, we have
\begin{eqnarray}\nonumber
\lambda_{k+1}&\leq&\frac{1}{2k}(2+\frac{1}{n^{2}})\sum_{i=1}^{k}\lambda_{i}\\\label{ab}&&+\left\{
\frac{1}{k^{2}}(1+\frac{8}{n^{2}})^{2}(\sum_{i=1}^{k}\lambda_{i})^{2}-\frac{1}{k}(1+\frac{16}{n^{2}})\sum_{i=1}^{k}\lambda_{i}^{2}
  \right\}^{\frac{1}{2}},
\end{eqnarray}
and
\begin{equation}\label{ab1}
\lambda_{2}\leq (1+\frac{17}{2n^{2}})\lambda_{1}.
\end{equation}
\end{corollary}
\section{Preliminaries}
In this section, we describe the necessary tools about tensor  $T$ and  problem (\ref{2}) which enable us to prove our results. Throughout the paper, for any vector fields $X,Y$,  we  denote $\langle T(X), Y\rangle$ with $T(X,Y)$. For  any  $u,v\in C^{\infty}(\Omega)$, straightforward computation implies that
\begin{equation*}
\mathcal{L}(uv)=v\mathcal{L}u+u\mathcal{L}v+2T(\nabla u, \nabla v).
\end{equation*}
Let $d\mu$ be the volume form on the boundary induced by the outward normal vector field $\nu$ on $\partial \Omega$.  The divergence theorem for operator $\mathcal{L}$ as follows
\begin{equation*}
\int_{\Omega}\mathcal{L}u \,dm=\int_{\partial \Omega}T(\nabla u, \nu)d\mu,
\end{equation*}
then the integration by parts yields
\begin{equation*}
\int_{\Omega}v\mathcal{L}u \,dm=-\int_{\Omega}T(\nabla u, \nabla v)\,dm+\int_{\partial \Omega}vT(\nabla u, \nu)d\mu.
\end{equation*}
Hence, the operators  $\mathcal{L}$ and $\mathcal{L}^{2}$ are self-adjoint operator in the space of all function in $L^{2}(\Omega, dm)$ that vanish on $\partial \Omega$. Therefore   the eigenvalues of  problem (\ref{2}) are real and discrete.
\begin{proposition}
Let $\Omega$ be  a domain in a an $n$-dimensional complete Riemannian manifold $(M,\langle, \rangle)$,   $\lambda_{i}$ be  the $i$th eigenvalue of  (\ref{2}) and $u_{i}$  be the corresponding orthonormal real-valued eigenfunction. Then for any $h\in C^{4}(\Omega)\cup C^{3}(\partial \Omega)$ and any positive integer $k$, we have
\begin{eqnarray}\label{5}
\sum_{i=1}^{k}(\lambda_{k+1}-\lambda_{i})^{2}w_{i}&\leq&  \sum_{i=1}^{k}(\lambda_{k+1}-\lambda_{i})||p_{i}||^{2},\\\label{6}
\sum_{i=1}^{k}(\lambda_{k+1}-\lambda_{i})^{2}v_{i}&\leq&\delta\sum_{i=1}^{k}(\lambda_{k+1}-\lambda_{i})^{2}w_{i}\\\nonumber&&+\frac{1}{\delta}\sum_{i=1}^{k}(\lambda_{k+1}-\lambda_{i})||T(\nabla h, \nabla u_{i})+\frac{u_{i}\mathcal{L}h}{2}||^{2},\\\label{7}
\sum_{i=1}^{k}(\lambda_{k+1}-\lambda_{i})^{2}v_{i}&\leq&\delta\sum_{i=1}^{k}(\lambda_{k+1}-\lambda_{i})||p_{i}||^{2}\\\nonumber&&+\frac{1}{\delta}\sum_{i=1}^{k}(\lambda_{k+1}-\lambda_{i})||T(\nabla h, \nabla u_{i})+\frac{u_{i}\mathcal{L}h}{2}||^{2},
\end{eqnarray}
where $\delta$ is any positive constant,
\begin{eqnarray*}
w_{i}&=&\int_{\Omega}hu_{i}p_{i}\,dm,\\
p_{i}&=&\mathcal{L}h\mathcal{L}u_{i}+2T(\nabla h, \nabla \mathcal{L}u_{i})+\mathcal{L}(u_{i}\mathcal{L}h)+2\mathcal{L}T(\nabla h, \nabla u_{i}),\\
v_{i}&=&\int_{\Omega}u_{i}^{2}T(\nabla h, \nabla h)\,dm.
\end{eqnarray*}
\end{proposition}
\begin{proof}
For each $i$, $1\leq i\leq k$, consider the functions $\phi_{i}:\Omega \to \mathbb{R}$ given by
\begin{equation}\label{8}
\phi_{i}=hu_{i}-\sum_{j=1}^{k}a_{ij}u_{j}
\end{equation}
where $a_{ij}=\int_{\Omega}hu_{i}u_{j}\,dm$. We have $\phi_{i}|_{\partial \Omega}=\frac{\partial \phi_{i}}{\partial \nu}|_{\partial \Omega}=0$ and
\begin{equation*}
\int_{\Omega}\phi_{i}u_{r} \,dm=\int_{\Omega}hu_{i}u_{r}\,dm-\sum_{j=1}^{k}a_{ij}\int_{ \Omega}u_{r}u_{j}d\mu=0,\,\,\,\,\,\,\,\forall i, r=1,2,...,k.
\end{equation*}
Then by the inequality of Rayleigh-Ritz, we get
\begin{equation}\label{9}
\lambda_{k+1}\leq\frac{\int_{\Omega}\phi_{i}\mathcal{L}^{2}\phi_{i}dm}{\int_{\Omega}\phi_{i}^{2}dm},\,\,\,\,\,\,\,\forall i=1,2,...,k.
\end{equation}
Since
\begin{equation*}
\mathcal{L}\phi_{i}=\mathcal{L}(hu_{i})-\sum_{j=1}^{k}a_{ij}\mathcal{L}u_{j}=h\mathcal{L}u_{i}+u_{i}\mathcal{L}h+2T(\nabla h,\nabla u_{i})-\sum_{j=1}^{k}a_{ij}\mathcal{L}u_{j},
\end{equation*}
we obtain
\begin{equation*}
\mathcal{L}^{2}\phi_{i}=\mathcal{L}h\mathcal{L}u_{i}+\lambda_{i}hu_{i}+2T(\nabla h,\nabla \mathcal{L}u_{i})+\mathcal{L}(u_{i}\mathcal{L}h)+2\mathcal{L}T(\nabla h,\nabla u_{i})-\sum_{j=1}^{k}a_{ij}\lambda_{j}u_{j},
\end{equation*}
therefore we get
\begin{equation}\label{10}
\int_{\Omega}\phi_{i}\mathcal{L}^{2}\phi_{i}dm=\lambda_{i}||\phi_{i}||^{2}+\int_{\Omega}hu_{i}p_{i}\,dm-\sum_{j=1}^{k}a_{ij}r_{ij},
\end{equation}
where $r_{ij}=\int_{\Omega}p_{i}u_{j}dm$ and
\begin{equation*}
p_{i}=\mathcal{L}h\mathcal{L}u_{i}+2T(\nabla h,\nabla \mathcal{L}u_{i})+\mathcal{L}(u_{i}\mathcal{L}h)+2\mathcal{L}T(\nabla h,\nabla u_{i}).
\end{equation*}
Using integration by parts, we deduce that
\begin{eqnarray}\label{11}
&&\int_{\Omega}u_{j}\mathcal{L}T(\nabla h,\nabla u_{i})dm+\int_{\Omega}u_{j}T(\nabla h,\nabla \mathcal{L}u_{i})dm\\\nonumber
&&\,\,\,=-\int_{\Omega}T(\nabla u_{j},\nabla T(\nabla h,\nabla u_{i}))dm-\int_{\Omega}div(u_{j}T\nabla h) \mathcal{L}u_{i}dm\\\nonumber
&&\,\,\,=\int_{\Omega}\mathcal{L} u_{j}T(\nabla h,\nabla u_{i})dm-\int_{\Omega}\mathcal{L} u_{i}T(\nabla h,\nabla u_{j})dm-\int_{\Omega}u_{j}\mathcal{L}h \mathcal{L} u_{i}dm.
\end{eqnarray}
On the other hand
\begin{eqnarray}
&&\int_{\Omega}\mathcal{L} u_{j}T(\nabla h,\nabla u_{i})dm-\int_{\Omega}\mathcal{L} u_{i}T(\nabla h,\nabla u_{j})dm\\\nonumber
&&\,\,\,=-\int_{\Omega}h\,div(\mathcal{L}u_{j}T\nabla u_{i})+\int_{\Omega}h\,div(\mathcal{L}u_{i}T\nabla u_{j})
\\\nonumber
&&\,\,\,=-\int_{\Omega}\langle h \nabla \mathcal{L} u_{j},T\nabla u_{i} \rangle dm+\int_{\Omega}\langle h \nabla \mathcal{L}u_{i},T\nabla u_{j} \rangle dm
\\\nonumber
&&\,\,\,=\int_{\Omega}u_{i}div(hT\nabla u_{j})dm-\int_{\Omega}u_{j}div(hT\nabla u_{i})dm
\\\nonumber
&&\,\,\,=\int_{\Omega}(u_{i}h \mathcal{L}^{2} u_{j}-u_{j}h \mathcal{L}^{2} u_{i} ) \,dm+\int_{\Omega}\left(\langle u_{i}T \nabla h,\nabla  \mathcal{L} u_{j} \rangle -\langle u_{j}T \nabla h,\nabla  \mathcal{L} u_{i} \rangle\right) \,dm
\\\nonumber
&&\,\,\,=(\lambda_{j}-\lambda_{i})a_{ij}-\int_{\Omega}\mathcal{L} u_{j}T(\nabla h,\nabla u_{i})\,dm+\int_{\Omega}\mathcal{L} u_{i}T(\nabla h,\nabla u_{j})\,dm\\\nonumber&&\,\,\,\,\,\,\,\,\,\,-
\int_{\Omega}u_{i}\mathcal{L} u_{j}\mathcal{L} h\,dm+\int_{\Omega}u_{j}\mathcal{L} u_{i}\mathcal{L} h\,dm,
\end{eqnarray}
which implies that
\begin{eqnarray}\label{12}
&&2\int_{\Omega}\mathcal{L} u_{j}T(\nabla h,\nabla u_{i})dm-2\int_{\Omega}\mathcal{L} u_{i}T(\nabla h,\nabla u_{j})dm\\\nonumber&&\,\,\,=(\lambda_{j}-\lambda_{i})a_{ij}-
\int_{\Omega}u_{i}\mathcal{L} u_{j}\mathcal{L} h\,dm+\int_{\Omega}u_{j}\mathcal{L} u_{i}\mathcal{L} h\,dm.
\end{eqnarray}
Substituting (\ref{11}) into (\ref{12}), we have
\begin{eqnarray}\label{13}
&&2\int_{\Omega}u_{j}\mathcal{L}T(\nabla h,\nabla u_{i})dm+2\int_{\Omega}u_{j}T(\nabla h,\nabla \mathcal{L}u_{i})dm\\\nonumber
&&\,\,\,=(\lambda_{j}-\lambda_{i})a_{ij}-
\int_{\Omega}u_{i}\mathcal{L} u_{j}\mathcal{L} h\,dm-\int_{\Omega}u_{j}\mathcal{L} u_{i}\mathcal{L} h\,dm.
\end{eqnarray}
Moreover
\begin{equation}\label{14}
\int_{\Omega}u_{j}\mathcal{L}(u_{i}\mathcal{L}h)\,dm=\int_{\Omega}u_{i}\mathcal{L} u_{j}\mathcal{L} h\,dm.
\end{equation}
Combining (\ref{13}), (\ref{14}) and  $r_{ij}=\int_{\Omega}p_{i}u_{j}dm$, we can write
\begin{equation}\label{15}
r_{ij}=(\lambda_{j}-\lambda_{i})a_{ij}.
\end{equation}
It follows from (\ref{9}), (\ref{10}) and (\ref{11}) that
\begin{eqnarray}\label{16}
(\lambda_{k+1}-\lambda_{i})||\phi_{i}||^{2}&\leq&\int_{\Omega}\phi_{i}\mathcal{L}^{2}\phi_{i}dm-\lambda_{i}||\phi_{i}||^{2}\\\nonumber
&\leq&\int_{\Omega}\phi_{i}p_{i}\,dm=w_{i}+\sum_{j=1}^{k}(\lambda_{i}-\lambda_{j})a_{ij}^{2},
\end{eqnarray}
where $w_{i}=\int_{\Omega}hu_{i}p_{i}\,dm$. We use that $\int_{\Omega}\phi_{i}u_{j}\,dm=0$ again to get
\begin{eqnarray}\nonumber
(\lambda_{k+1}-\lambda_{i})\left( \int_{\Omega}\phi_{i}p_{i}\,dm\right)^{2}&=&(\lambda_{k+1}-\lambda_{i}))\left( \int_{\Omega}\phi_{i}(p_{i}-\sum_{j=1}^{k}r_{ij}u_{j}\,dm\right)^{2}\\\label{17}
&\leq&(\lambda_{k+1}-\lambda_{i})||\phi_{i}||^{2}\left(||p_{i}||^{2}-\sum_{j=1}^{k}r_{ij}^{2} \right)
\\\nonumber
&\leq&\left( \int_{\Omega}\phi_{i}p_{i}\,dm\right)\left(||p_{i}||^{2}-\sum_{j=1}^{k}r_{ij}^{2} \right),
\end{eqnarray}
this implies that
\begin{equation}\label{18}
(\lambda_{k+1}-\lambda_{i})\left( \int_{\Omega}\phi_{i}p_{i}\,dm\right)\leq||p_{i}||^{2}-\sum_{j=1}^{k}r_{ij}^{2}.
\end{equation}
Multiplying (\ref{18}) by $(\lambda_{k+1}-\lambda_{i})$ and summing on $i$ from $1$ to $k$, we obtain
\begin{equation}\label{19}
\sum_{i=1}^{k}(\lambda_{k+1}-\lambda_{i})^{2} \int_{\Omega}\phi_{i}p_{i}\,dm\leq\sum_{i=1}^{k}(\lambda_{k+1}-\lambda_{i})||p_{i}||^{2}-\sum_{i,j=1}^{k}(\lambda_{k+1}-\lambda_{i})(\lambda_{i}-\lambda_{j})^{2}a_{ij}^{2}.
\end{equation}
Multiplying (\ref{16}) by $(\lambda_{k+1}-\lambda_{i})^{2}$, summing on $i$ from $1$ to $k$ and $a_{ij}=a_{ji}$, we infer
\begin{eqnarray}\nonumber
\sum_{i=1}^{k}(\lambda_{k+1}-\lambda_{i})^{2}\int_{\Omega}\phi_{i}p_{i}\,dm&=&\sum_{i=1}^{k}(\lambda_{k+1}-\lambda_{i})^{2}w_{i}+\sum_{i,j=1}^{k}(\lambda_{k+1}-\lambda_{i})^{2}(\lambda_{i}-\lambda_{j})a_{ij}^{2}\\
&=&\sum_{i=1}^{k}(\lambda_{k+1}-\lambda_{i})^{2}w_{i}-\sum_{i,j=1}^{k}(\lambda_{k+1}-\lambda_{i})(\lambda_{i}-\lambda_{j})^{2}a_{ij}^{2},
\end{eqnarray}
then
\begin{equation}\label{20}
\sum_{i=1}^{k}(\lambda_{k+1}-\lambda_{i})^{2}w_{i}\leq\sum_{i=1}^{k}(\lambda_{k+1}-\lambda_{i})||p_{i}||^{2},
\end{equation}
which shows that  (\ref{5}) is true. In order to prove (\ref{6}), we set
\begin{equation}\label{21}
b_{ij}=\int_{\Omega}u_{j}\left(T(\nabla h,\nabla u_{i})+\frac{u_{i}}{2}\mathcal{L}h\right)dm.
\end{equation}
Observe that
\begin{eqnarray*}
b_{ij}&=&-\int_{\Omega}u_{i}div(u_{j}T\nabla h)dm+\frac{1}{2}\int_{\Omega}u_{i}u_{j}\mathcal{L}hdm\\&=&
-\int_{\Omega}u_{i}\left(T(\nabla h,\nabla u_{j})+\frac{u_{j}}{2}\mathcal{L}h\right)dm=-b_{ji},
\end{eqnarray*}
and
\begin{equation}\label{22}
-2\int_{\Omega}\phi_{i}\left(T(\nabla h,\nabla u_{i})+\frac{u_{i}}{2}\mathcal{L}h\right)dm=v_{i}+2\sum_{j=1}^{k}a_{ij}b_{ij},
\end{equation}
where
\begin{equation*}
v_{i}=\int_{\Omega}\left(-hu_{i}^{2}\mathcal{L}{h}-2hu_{i}T(\nabla h,\nabla u_{i})\right)dm=\int_{\Omega}u_{i}^{2}T(\nabla h, \nabla h)\,dm.
\end{equation*}
Since for any positive constant  $\delta$ and for all  $x,y\in \mathbb{R}$ we have $-2xy\leq \delta x^{2}+\frac{y^{2}}{\delta}$, then multiplying (\ref{22}) by $(\lambda_{k+1}-\lambda_{i})^{2}$, for any positive constant  $\delta$, we get
\begin{eqnarray*}
&&(\lambda_{k+1}-\lambda_{i})^{2}(v_{i}+2\sum_{j=1}^{k}a_{ij}b_{ij} )\\\,\,&&=(\lambda_{k+1}-\lambda_{i})^{2}
\int_{\Omega}(-2\phi_{i})(T(\nabla h,\nabla u_{i})+\frac{u_{i}}{2}\mathcal{L}h-\sum_{j=1}^{k}b_{ij}u_{j})dm\\\,\,&&\leq
\delta(\lambda_{k+1}-\lambda_{i})^{3}||\phi_{i}||^{2}+\frac{1}{\delta}(\lambda_{k+1}-\lambda_{i})
\int_{\Omega}(T(\nabla h,\nabla u_{i})+\frac{u_{i}}{2}\mathcal{L}h-\sum_{j=1}^{k}b_{ij}u_{j})dm\\\,\,&&\leq
\delta(\lambda_{k+1}-\lambda_{i})^{3}||\phi_{i}||^{2}+\frac{1}{\delta}(\lambda_{k+1}-\lambda_{i})
(|| T(\nabla h,\nabla u_{i})+\frac{u_{i}}{2}\mathcal{L}h||^{2}-\sum_{j=1}^{k}b_{ij}^{2}),
\end{eqnarray*}
hence (\ref{16}) implies that
\begin{eqnarray*}
(\lambda_{k+1}-\lambda_{i})^{2}(v_{i}+2\sum_{j=1}^{k}a_{ij}b_{ij} )&\leq&
\delta(\lambda_{k+1}-\lambda_{i})^{2}(w_{i}+\sum_{j=1}^{k}(\lambda_{i}-\lambda_{j})a_{ij}^{2})\\\,\,&&+\frac{1}{\delta}(\lambda_{k+1}-\lambda_{i})
(|| T(\nabla h,\nabla u_{i})+\frac{u_{i}}{2}\mathcal{L}h||^{2}-\sum_{j=1}^{k}b_{ij}^{2}).
\end{eqnarray*}
Now,  summing over $i$ from $1$  to $k$, $a_{ij}=a_{ij}$ and $b_{ij}=-b_{ji}$ we conclude that
\begin{eqnarray*}
&&\sum_{i=1}^{k}(\lambda_{k+1}-\lambda_{i})^{2}v_{i}-2\sum_{i,j=1}^{k}(\lambda_{k+1}-\lambda_{i})(\lambda_{i}-\lambda_{j})a_{ij}b_{ij}\\ &&\leq\sum_{i=1}^{k}
\delta(\lambda_{k+1}-\lambda_{i})^{2}w_{i}-\sum_{i,j=1}^{k}\delta(\lambda_{k+1}-\lambda_{i})(\lambda_{i}-\lambda_{j})^{2}a_{ij}^{2}\\\,\,&&+\frac{1}{\delta}\sum_{i=1}^{k}(\lambda_{k+1}-\lambda_{i})|| T(\nabla h,\nabla u_{i})+\frac{u_{i}}{2}\mathcal{L}h||^{2}-
\frac{1}{\delta}\sum_{i,j=1}^{k}(\lambda_{k+1}-\lambda_{i})b_{ij}^{2},
\end{eqnarray*}
which gives
\begin{eqnarray}\nonumber
\sum_{i=1}^{k}(\lambda_{k+1}-\lambda_{i})^{2}v_{i}&\leq&\sum_{i=1}^{k}
\delta(\lambda_{k+1}-\lambda_{i})^{2}w_{i}+\frac{1}{\delta}\sum_{i=1}^{k}(\lambda_{k+1}-\lambda_{i})|| T(\nabla h,\nabla u_{i})+\frac{u_{i}}{2}\mathcal{L}h||^{2}\\ \label{23}&&-\sum_{i,j=1}^{k}(\lambda_{k+1}-\lambda_{i})(\sqrt{\delta}(\lambda_{i}-\lambda_{j})a_{ij}-\frac{1}{\delta}b_{ij})^{2}\\\nonumber&\leq&
\sum_{i=1}^{k}
\delta(\lambda_{k+1}-\lambda_{i})^{2}w_{i}+\frac{1}{\delta}\sum_{i=1}^{k}(\lambda_{k+1}-\lambda_{i})|| T(\nabla h,\nabla u_{i})+\frac{u_{i}}{2}\mathcal{L}h||^{2}.
\end{eqnarray}
Thus (\ref{6}) is true. Substituting (\ref{20}) into (\ref{23}) complete the  proof of the proposition.
\end{proof}
\begin{proof}[Proof of Theorem \ref{t1}]
Let $x_{1},..., x_{n}$ be the  standard Euclidean coordinate of $\mathbb{R}^{m}$, $\bar{\nabla}$ be the Canonical connection of $\mathbb{R}^{m}$ and  $\{e_{1},...,e_{m}\}$ be a local orthonormal geodesic frame in $p\in M$ adapted to $M$, then
\begin{equation*}
\bar{\nabla}x_{r}=\sum_{i=1}^{n}e_{i}(x_{r})e_{i}+\sum_{i=n+1}^{m}e_{i}(x_{r})e_{i},\qquad
e_{r}=\nabla x_{r}+(\nabla x_{r})^{\perp}.
\end{equation*}
Therefore

 \begin{equation*}
\sum_{r=1}^{m}T(\nabla x_{r}, \nabla u_{i})=\sum_{r=1}^{m}\langle \nabla x_{r}, T(\nabla u_{i})\rangle=
\sum_{r=1}^{m} \langle e_{r}-(\nabla x_{r})^{\perp},T(\nabla u_{i})\rangle=
\sum_{r=1}^{m} \langle e_{r},T(\nabla u_{i})\rangle,
\end{equation*}
and
 \begin{equation}\label{a1}
\sum_{r=1}^{m}T(\nabla x_{r}, \nabla u_{i})^{2}=
\sum_{r=1}^{m} \langle e_{r},T(\nabla u_{i})\rangle=|T(\nabla u_{i})|^{2}.
\end{equation}
Also, we have
 \begin{equation}\label{a2}
\sum_{r=1}^{m}T(\nabla x_{r}, \nabla x_{r})=
\sum_{r=1}^{m} \langle e_{r},T(\nabla  x_{r})\rangle=
\sum_{r=1}^{m} \langle T( e_{r}),\nabla  x_{r}\rangle=
\sum_{r=1}^{m} \langle T( e_{r}),e_{r}\rangle=tr(T).
\end{equation}
For $x=(x_{1},..., x_{n})$, we compute
\begin{eqnarray}\nonumber
div(T(\nabla x))&:=&\Big (div(T(\nabla x_{1})),...,div(T(\nabla x_{m}))\Big)\\\label{a6}&=&
\left(\sum_{i=1}^{n}e_{i}\langle T(\nabla  x_{1}),e_{i} \rangle,..., \sum_{i=1}^{n}e_{i}\langle T(\nabla  x_{m}),e_{i} \rangle\right)\\\nonumber&=&
\sum_{i,j=1}^{n}\Big(e_{i}(e_{j}(x_{1})\langle T(e_{j}),e_{i} \rangle),..., e_{i}(e_{j}(x_{m})\langle T(e_{j}),e_{i} \rangle)\Big)\\\nonumber&=&
\sum_{i,j=1}^{n}\Big(e_{i}e_{j}(x_{1})\langle T(e_{j}),e_{i} \rangle,..., e_{i}e_{j}(x_{m})\langle T(e_{j}),e_{i} \rangle\Big)\\\nonumber&&+
\sum_{i,j=1}^{n}\Big(e_{j}(x_{1})\langle\nabla_{ e_{i}} T(e_{j}),e_{i} \rangle,..., e_{j}(x_{m})\langle\nabla_{ e_{i}}  T(e_{j}),e_{i} \rangle\Big)\\\nonumber&=&
\sum_{i,j=1}^{n}\langle T(e_{j}),e_{i} \rangle\bar{\nabla}_{e_{i}}e_{j}(x)+\sum_{i,j=1}^{n}\langle \nabla_{e_{i}}T(e_{j}),e_{i} \rangle e_{j}(x)\\\nonumber&=&
\sum_{i,j=1}^{n}\langle T(e_{j}),e_{i} \rangle\alpha(e_{i},e_{j})(x)+\sum_{i,j=1}^{n}\langle \nabla_{e_{i}}T(e_{i}),e_{j} \rangle e_{j}(x)\\\nonumber&=&
\sum_{j=1}^{n}\alpha(T(e_{j}),e_{j})(x)+\sum_{j=1}^{n} \nabla_{e_{i}}T(e_{i}) (x)=tr(\alpha\circ T)(x)+tr(\nabla T)(x),
\end{eqnarray}
hence
 \begin{equation}\label{a3}
\sum_{r=1}^{m}(div(T(\nabla x_{r})))^{2}=||div T(\nabla x)||^{2}=||tr(\alpha\circ T)||^{2}+|tr(\nabla T)|^{2},
\end{equation}
and
\begin{eqnarray}\nonumber
\sum_{r=1}^{m}div(T(\nabla x_{r}))T(\nabla x_{r}, \nabla u_{i})&=&\sum_{r=1}^{m}div(T(\nabla x_{r}))T( \nabla u_{i})( x_{r})\\\label{a4}&=&\langle div(T(\nabla x)),T( \nabla u_{i}) \rangle=\langle tr(\nabla T),T( \nabla u_{i})\rangle,
\end{eqnarray}
where
$\alpha\circ  T=\alpha(T(.),.)\in \mathcal{X}(M)^{\perp}$. By taking $h=x_{r}$ in (\ref{6})  we can write
\begin{eqnarray}\nonumber
\sum_{i=1}^{k}(\lambda_{k+1}-\lambda_{i})^{2}v_{i}&\leq&\delta\sum_{i=1}^{k}(\lambda_{k+1}-\lambda_{i})^{2}w_{i}\\\label{a0}&&+\frac{1}{\delta}\sum_{i=1}^{k}(\lambda_{k+1}-\lambda_{i})||T(\nabla x_{r}, \nabla u_{i})+\frac{u_{i}\mathcal{L}x_{r}}{2}||^{2}.
\end{eqnarray}
where
\begin{eqnarray*}
w_{i}&=&\int_{\Omega}x_{r}u_{i}p_{i}\,dm,\\
p_{i}&=&\mathcal{L}x_{r}\mathcal{L}u_{i}+2T(\nabla x_{r}, \nabla \mathcal{L}u_{i})+\mathcal{L}(u_{i}\mathcal{L}h)+2\mathcal{L}T(\nabla x_{r}, \nabla u_{i}),\\
v_{i}&=&\int_{\Omega}u_{i}^{2}T(\nabla x_{r}, \nabla x_{r})\,dm.
\end{eqnarray*}
Summing over $r$, we have
\begin{eqnarray}\nonumber
\sum_{r=1}^{m}\sum_{i=1}^{k}(\lambda_{k+1}-\lambda_{i})^{2}v_{i}&\leq&\delta\sum_{r=1}^{m}\sum_{i=1}^{k}(\lambda_{k+1}-\lambda_{i})^{2}w_{i}\\\label{24}&&+\frac{1}{\delta}\sum_{i=1}^{k}(\lambda_{k+1}-\lambda_{i})\sum_{r=1}^{m}||T(\nabla x_{r}, \nabla u_{i})+\frac{u_{i}\mathcal{L}x_{r}}{2}||^{2},
\end{eqnarray}
and from (\ref{a2}) we get
\begin{equation}\label{a5}
\sum_{r=1}^{m} v_{i}=\sum_{r=1}^{m} \int_{\Omega}u_{i}^{2}T(\nabla x_{r}, \nabla x_{r})\,dm= \int_{\Omega}u_{i}^{2}\,tr(T)\,dm.
\end{equation}
Also (\ref{a3}) and   (\ref{a6})  imply that
\begin{eqnarray}\nonumber
\sum_{r=1}^{m}  \int_{\Omega}x_{r}u_{i}\mathcal{L} x_{r}\mathcal{L}u_{i}\,dm&=& \int_{\Omega}u_{i}\mathcal{L}u_{i}\,\langle divT(\nabla x),x \rangle\,dm\\\label{a7}&=& \int_{\Omega}u_{i}\mathcal{L}u_{i}\,\langle tr(\alpha\circ T)+tr(\nabla T),I \rangle\,dm,
\end{eqnarray}

\begin{eqnarray}
&&\sum_{r=1}^{m}  \int_{\Omega}x_{r}u_{i}\mathcal{L} (u_{i}\mathcal{L}x_{r})\,dm\\\nonumber&=&
\sum_{r=1}^{m}  \int_{\Omega}\mathcal{L} (x_{r}u_{i})u_{i}\mathcal{L}x_{r}\,dm\\\nonumber&=&
\sum_{r=1}^{m}  \int_{\Omega}\left( x_{r}u_{i}\mathcal{L} u_{i}\mathcal{L}x_{r}+u_{i}^{2}(\mathcal{L} x_{r})^{2}+2T(\nabla u_{i},\nabla x_{r}) u_{i}\mathcal{L} x_{r}\right)\,dm\\\nonumber&=&
\int_{\Omega}u_{i}\mathcal{L} u_{i}\,\langle tr(\alpha\circ T)+tr(\nabla T),I \rangle\,dm+\int_{\Omega}u_{i}^{2}(||tr(\alpha\circ T)||^{2}+|tr(\nabla T)|^{2})dm\\\nonumber&&+2\int_{\Omega}u_{i}\langle T(\nabla u_{i}),tr(\nabla T) \rangle\,dm,
\end{eqnarray}
and
\begin{eqnarray}
&&\sum_{r=1}^{m}  \int_{\Omega}2x_{r}u_{i}\mathcal{L}T (\nabla x_{r},\nabla u_{i})\,dm\\\nonumber&=&
2\sum_{r=1}^{m}  \int_{\Omega}\mathcal{L} (x_{r}u_{i})T (\nabla x_{r},\nabla u_{i})\,dm\\\nonumber&=&
2\sum_{r=1}^{m}  \int_{\Omega}u_{i}\mathcal{L}x_{r}T (\nabla x_{r},\nabla u_{i})\,dm+ 2\sum_{r=1}^{m}  \int_{\Omega}x_{r}\mathcal{L} u_{i}T (\nabla x_{r},\nabla u_{i})\,dm
\\\nonumber&&+2\sum_{r=1}^{m}  \int_{\Omega}T (\nabla x_{r},\nabla u_{i})^{2}\,dm\\\nonumber
&=&2\int_{\Omega}u_{i}\langle T(\nabla u_{i}),tr(\nabla T) \rangle\,dm+4\int_{\Omega}|T(\nabla u_{i})|^{2}\,dm+2\int_{\Omega}\mathcal{L}u_{i}\,
\langle T(\nabla u_{i}),I\rangle\,dm
\end{eqnarray}
\begin{equation}\label{a8}
\sum_{r=1}^{m} \int_{\Omega}2x_{r}u_{i}T(\nabla x_{r}, \nabla \mathcal{L} u_{i})\,dm= 2\int_{\Omega}u_{i}\langle T(\nabla \mathcal{L}u_{i}),I \rangle\,dm.
\end{equation}
Thus
\begin{eqnarray}\nonumber
\sum_{r=1}^{m}w_{i}&=&2 \int_{\Omega}u_{i}\mathcal{L}u_{i}\,\langle tr(\alpha\circ T)+tr(\nabla T),I \rangle\,dm+2 \int_{\Omega}u_{i}\langle T(\nabla \mathcal{L}u_{i}),I \rangle\,dm\\\label{a9}&&+\int_{\Omega}u_{i}^{2}(||tr(\alpha\circ T)||^{2}+|tr(\nabla T)|^{2})dm+4\int_{\Omega}u_{i}\langle T(\nabla u_{i}),tr(\nabla T) \rangle\,dm\\\nonumber&&+
4\int_{\Omega}|T(\nabla u_{i})|^{2}\,dm+2\int_{\Omega}\mathcal{L}u_{i}\,
\langle T(\nabla u_{i}),I\rangle\,dm
\end{eqnarray}
and
\begin{eqnarray}\label{a10}
&&\sum_{r=1}^{m}  ||T(\nabla x_{r}, \nabla u_{i})+\frac{u_{i}\mathcal{L}x_{r}}{2}||^{2}\\\nonumber&=&
\sum_{r=1}^{m}  \int_{\Omega}\left( T (\nabla x_{r},\nabla u_{i})^{2}+u_{i}\mathcal{L}x_{r} T (\nabla x_{r},\nabla u_{i})+\frac{1}{4} u_{i}^{2}(\mathcal{L} x_{r})^{2}
\right)dm
\\\nonumber&=& \int_{\Omega}
\left( |T(\nabla u_{i})|^{2}+u_{i}\langle T(\nabla u_{i}),tr(\nabla T) \rangle+\frac{1}{4} u_{i}^{2}(||tr(\alpha\circ T)||^{2}+|tr(\nabla T)|^{2})
\right)dm.
\end{eqnarray}
Substituting (\ref{a5}), (\ref{a9}) and (\ref{a10}) into (\ref{a0}) we complete the proof of the theorem.
\end{proof}


\begin{proof}[Proof of Theorem \ref{t2}]
Let $S_{e_{i}}$ be the Weingarten operator of the immersion with respect to $e_{i}$. Then
\begin{eqnarray*}
||tr(\alpha\circ T)||^{2}&=&||\sum_{i=1}^{n}\alpha(Te_{i},e_{i})||^{2}=||\sum_{i=1}^{n}\sum_{k=n+1}^{m}\langle\alpha(Te_{i},e_{i}),e_{k}\rangle e_{k}||^{2}\\
&=&||\sum_{k=n+1}^{m}(\sum_{i=1}^{n}\langle S_{e_{k}}e_{i},Te_{i}\rangle )e_{k}||^{2}=
||\sum_{k=n+1}^{m}\langle S_{e_{k}},T\rangle e_{k}||^{2}\\&\leq&\sum_{k=n+1}^{m}|\langle S_{e_{k}},T\rangle|^{2}\sum_{k=n+1}^{m}|e_{k}|^{2}\leq \sum_{k=n+1}^{m}|S_{e_{k}}|^{2}|T|^{2}\sum_{k=n+1}^{m}|e_{k}|^{2}\\&\leq&
(m-n)S_{0}^{2}T_{*}^{2},
\end{eqnarray*}
where $S_{0}=\max\{\sup_{\bar{\Omega}}|S_{e_{k}}|: \,\,k=n+1,...,m\}$ and $T_{*}=\sup_{\bar{\Omega}} |T|$. If $T_{0}=\sup_{\bar{\Omega}}|tr(\nabla T)|$ and $I_{0}=\sup_{\bar{\Omega}}|I|$  then
\begin{eqnarray}\nonumber
\int_{\Omega}u_{i}\langle T(\nabla u_{i}),tr(\nabla T) \rangle\,dm&\leq& \left(\int_{\Omega} u_{i}^{2}dm\right)^{\frac{1}{2}}
 \left(\int_{\Omega} |T(\nabla u_{i})|^{2}|tr(\nabla T)|^{2}dm\right)^{\frac{1}{2}}\\\label{b1}&\leq&T_{0}||T(\nabla u_{i})||_{L^{2}(\Omega)},
\end{eqnarray}

\begin{eqnarray}\nonumber
&&\int_{\Omega}u_{i}\mathcal{L}u_{i}\langle tr(\alpha\circ T)+tr(\nabla T), I\rangle dm\\\nonumber&&=
\int_{\Omega}u_{i}\mathcal{L}u_{i}\langle tr(\alpha\circ T), I\rangle dm+
\int_{\Omega}u_{i}\mathcal{L}u_{i}\langle tr(\nabla T), I\rangle dm\\\nonumber
&&\leq \left(\int_{\Omega} u_{i}^{2}dm\right)^{\frac{1}{2}} \left(\int_{\Omega}(\mathcal{L}u_{i})^{2} ||tr(\alpha\circ T)|^{2}|I|^{2}dm\right)^{\frac{1}{2}}\\\nonumber
&&\,\,\,+\left(\int_{\Omega} u_{i}^{2}dm\right)^{\frac{1}{2}} \left(\int_{\Omega}(\mathcal{L}u_{i})^{2} |tr(\nabla T)|^{2}|I|^{2}dm\right)^{\frac{1}{2}}\\\label{c1}
&&\leq(\sqrt{m-n} S_{0}T_{*}+T_{0})I_{0}\lambda_{i}^{\frac{1}{2}}
\end{eqnarray}
and
\begin{eqnarray}\nonumber
\int_{\Omega}u_{i}\langle T(\nabla \mathcal{L}u_{i}),I \rangle\,dm&\leq& \left(\int_{\Omega} u_{i}^{2}dm\right)^{\frac{1}{2}}
 \left(\int_{\Omega} |T(\nabla \mathcal{L}u_{i})|^{2}|I|^{2}dm\right)^{\frac{1}{2}}\\\label{c2}&\leq&I_{0}||T(\nabla\mathcal{L}u_{i})||_{L^{2}(\Omega)}.
\end{eqnarray}
Also, we have
\begin{eqnarray}\nonumber
\int_{\Omega}u_{i}^{2}(||tr(\alpha\circ T)||^{2}+|tr(\nabla T)|^{2})dm\leq (m-n)S_{0}^{2}T_{*}^{2}+T_{0}^{2}
\end{eqnarray}
and
\begin{eqnarray}\nonumber
\int_{\Omega}\mathcal{L}u_{i}\langle T(\nabla u_{i}),I\rangle dm&\leq&\left (\int_{\Omega}(\mathcal{L}u_{i})^{2}dm\right)^{\frac{1}{2}}\left(\int_{\Omega}|T(\nabla u_{i})|^{2}|I|^{2}dm\right)^{\frac{1}{2}}\\&\leq&\lambda_{i}
 I_{0}||T(\nabla u_{i})||_{L^{2}(\Omega)}.
\end{eqnarray}
By setting
\begin{eqnarray*}
C_{i}&=&2(\sqrt{m-n} S_{0}T_{*}+T_{0})I_{0}\lambda_{i}^{\frac{1}{2}}+I_{0}||T(\nabla\mathcal{L}u_{i})||_{L^{2}(\Omega)}+ (m-n)S_{0}^{2}T_{*}^{2}+T_{0}^{2}\\&&+4 T_{0}||T(\nabla u_{i})||_{L^{2}(\Omega)}+4 ||T(\nabla u_{i})||_{L^{2}(\Omega)}+2\lambda_{i}
 I_{0}||T(\nabla u_{i})||_{L^{2}(\Omega)},
\end{eqnarray*}
and
\begin{equation*}
D_{i}=||T(\nabla u_{i})||_{L^{2}(\Omega)}+T_{0}||T(\nabla u_{i})||_{L^{2}(\Omega)}+\frac{1}{4}((m-n)S_{0}^{2}T_{*}^{2}+T_{0}^{2})
\end{equation*}
we get $A_{i}\leq C_{i}$ and $B_{i}\leq D_{i}$. Substituting these  inequalities into Theorem \ref{t1}  we complete the proof of the Theorem.
\end{proof}

\begin{proof}[Proof of Theorem \ref{t3}]
 Taking  $T$ equal to identity in  Theorem $1$, we obtain
\begin{equation}\label{dd}
n\sum_{i=1}^{k}(\lambda_{k+1}-\lambda_{i})^{2}\leq \delta \sum_{i=1}^{k}(\lambda_{k+1}-\lambda_{i})^{2}E_{i}+\frac{1}{\delta}
\sum_{i=1}^{k}(\lambda_{k+1}-\lambda_{i})F_{i},
\end{equation}where
\begin{eqnarray*}
E_{i}&=&2\int_{\Omega}u_{i}\Delta u_{i}\langle nH, I\rangle dm+2\int_{\Omega}u_{i}\langle \nabla \Delta u_{i}, I\rangle dm\\&&+\int_{\Omega}u_{i}^{2}n^{2}||H||^{2}dm+4\int_{\Omega}|\nabla u_{i}|^{2}dm+2\int_{\Omega}\Delta u_{i}\langle \nabla u_{i},I\rangle dm,
\end{eqnarray*}
and
\begin{equation*}
F_{i}=\int_{\Omega}\left\{ |\nabla u_{i}|^{2}+\frac{u_{i}^{2}}{4}n^{2} ||H||^{2}\right\}dm.
\end{equation*}
On the other hand, we have
\begin{equation*}
\int_{\Omega}u_{i}\langle \nabla \Delta u_{i}, I\rangle dm=-\int_{\Omega}\Delta u_{i}\langle \nabla u_{i},I\rangle dm
\end{equation*}
and
\begin{eqnarray*}
\int_{\Omega}|\nabla u_{i}|^{2} dm&=&\int_{\Omega}\langle \nabla u_{i},\nabla u_{i}\rangle dm=-\int_{\Omega} u_{i}\Delta u_{i} dm\\&\leq&
\int_{\Omega} |u_{i}|  |\Delta u_{i}| dm\leq
(\int_{\Omega} u_{i}^{2} dm)^{\frac{1}{2}}(\int_{\Omega}(\Delta u_{i})^{2} dm)^{\frac{1}{2}}\leq
\lambda_{i}^{\frac{1}{2}}.
\end{eqnarray*}
Hence
\begin{equation*}
E_{i}\leq 2nH_{0}I_{0}\lambda_{i}^{\frac{1}{2}}+n^{2}H_{0}^{2}+4\lambda_{i}^{\frac{1}{2}}
\end{equation*}
and
\begin{equation*}
F_{i}\leq \lambda_{i}^{\frac{1}{2}}+\frac{1}{4}n^{2}H_{0}^{2}
\end{equation*}
where $H_{0}=\sup_{\bar{\Omega}}|H|$ and $I_{0}=\sup_{\bar{\Omega}}|I|$.
Substituting these inequality  into (\ref{dd}) we complete the proof of the theorem.
\end{proof}
\begin{proof}[Proof of Corollary \ref{t4}]
 For a minimal hypersurface we have $H=0$, therefore Theorem 3 results that
\begin{equation}
n\sum_{i=1}^{k}(\lambda_{k+1}-\lambda_{i})^{2}\leq 4\delta \sum_{i=1}^{k}(\lambda_{k+1}-\lambda_{i})^{2}\lambda_{i}^{\frac{1}{2}}
+\frac{1}{\delta}
\sum_{i=1}^{k}(\lambda_{k+1}-\lambda_{i})\lambda_{i}^{\frac{1}{2}}.
\end{equation}
Taking
\begin{equation*}
\delta=\left\{\frac{\sum_{i=1}^{k}(\lambda_{k+1}-\lambda_{i})\lambda_{i}^{\frac{1}{2}}}{4 \sum_{i=1}^{k}(\lambda_{k+1}-\lambda_{i})^{2}\lambda_{i}^{\frac{1}{2}}}\right\}^{\frac{1}{2}},
\end{equation*}
we get
\begin{equation}\label{cc}
\frac{n}{4}\sum_{i=1}^{k}(\lambda_{k+1}-\lambda_{i})^{2}\leq
\left\{{ \sum_{i=1}^{k}(\lambda_{k+1}-\lambda_{i})^{2}\lambda_{i}^{\frac{1}{2}}}{\sum_{i=1}^{k}(\lambda_{k+1}-\lambda_{i})\lambda_{i}^{\frac{1}{2}}}\right\}^{\frac{1}{2}}.
\end{equation}
On the other hand
\begin{equation*}
\left(\sum_{i=1}^{k}(\lambda_{k+1}-\lambda_{i})^{2}\lambda_{i}^{\frac{1}{2}} \right)
\left( \sum_{i=1}^{k}(\lambda_{k+1}-\lambda_{i})\lambda_{i}^{\frac{1}{2}}\right)\leq
\left(\sum_{i=1}^{k}(\lambda_{k+1}-\lambda_{i})^{2} \right)
\left( \sum_{i=1}^{k}(\lambda_{k+1}-\lambda_{i})\lambda_{i}\right).
\end{equation*}
It and (\ref{cc}) imply that
\begin{equation}\label{cc1}
\sum_{i=1}^{k}(\lambda_{k+1}-\lambda_{i})^{2}\leq\frac{16}{n^{2}} \sum_{i=1}^{k}(\lambda_{k+1}-\lambda_{i})\lambda_{i}
\end{equation}
solving this quadratic polynomial of $\lambda_{k+1}$, we obtain (\ref{ab}) and $(\ref{ab1})$.
\end{proof}

\end{document}